\documentclass{amsart}

\newtheorem{theorem}{Theorem}
\newtheorem{lemma}[theorem]{Lemma}

\newtheorem{corollary}[theorem]{Corollary}

\theoremstyle{definition}

\usepackage{amscd,amssymb}

\begin{document}

\title[Inner automorphisms of Lie algebras of symmetric polynomials]
{Inner automorphisms of Lie algebras \\ of symmetric polynomials}

\author[{\c S}ehmus F{\i}nd{\i}k, Nazar {\c S}ah\.{i}n \"O{\u g}\"U{\c s}l\"u]
{{\c S}ehmus F{\i}nd{\i}k and Nazar {\c S}ah\.{i}n \"O{\u g}\"U{\c s}l\"u}

\address{Department of Mathematics,
\c{C}ukurova University, 01330 Balcal\i,
 Adana, Turkey}
\email{sfindik@cu.edu.tr}
\email{noguslu@cu.edu.tr}
\thanks
{}


\begin{abstract}
Let $L_{n}$ be the free Lie algebra,
$F_{n}$ be the free metabelian Lie algebra,
and $L_{n,c}$ be the free metabelian nilpotent of class $c$ Lie algebra of rank $n$
generated by $x_1,\ldots,x_n$  over a field $K$ of characteristic zero.
We call a polynomial $p(X_n)$ symmetric in these Lie algebras if
$p(x_1,\ldots,x_n)=p(x_{\pi(1)},\ldots,x_{\pi(n)})$ for each element $\pi$ 
of the symmetric group $S_n$.
The sets $L_n^{S_n}$,  $F_n^{S_n}$, and  $L_{n,c}^{S_n}$ of symmetric polynomials
coincide with the algebras of invariants of the group $S_n$ in $L_{n}$, $F_{n}$, and $L_{n,c}$, respectively.
We determine the groups $\text{\rm Inn}(F_{n}^{S_n})$ and $\text{\rm Inn}(L_{n,c}^{S_n})$ of inner automorphisms of the algebras $F_{n}^{S_n}$
and $L_{n,c}^{S_n}$, respectively. In particular, we obtain the descriptions of 
the groups $\text{\rm Aut}(L_{2}^{S_2})$, $\text{\rm Aut}(F_{2}^{S_2})$, and $\text{\rm Aut}(L_{2,c}^{S_2})$ of
all automorphisms of the algebras $L_{2}^{S_2}$, $F_{2}^{S_2}$, and $L_{2,c}^{S_2}$, respectively.
\end{abstract}

\maketitle

\section*{Introduction}

Let $A_n$ be the free algebra of rank $n$ over a field $K$ of characteristic zero in a variety of algebras
generated by $X_n=\{x_1,\ldots,x_n\}$. A polynomial $p(X_n)\in A_n$ is said to be symmetric if
$p(x_1,\ldots,x_n)=p(x_{\pi(1)},\ldots,x_{\pi(n)})$ for all $\pi\in S_n$. The set
of such polynomials is equal to the algebra $A_n^{S_n}$ of invariants of the symmetric group $S_n$.
The algebra $A_n^{S_n}$ is well known, when $A_n=K[X_n]$ is the
commutative associative unitary algebra by the fundamental theorem on symmetric polynomials:
\[
K[X_n]^{S_n}=K[\sigma_1,\ldots,\sigma_n], \ \ \sigma_i=x_1^i+\cdots+x_n^i, \ \ i=1,\ldots,n.
\]
For the case $A_n=K\left\langle X_{n}\right\rangle$, the associative algebra of rank $n$, see e.g. \cite{Wo}.
Now let $A_n=F_n$ be the free metabelian Lie algebra of rank $n$ over $K$.
It is well known, see e.g. \cite{Dr}, that the algebra $F_n^{S_n}$ of symmetric polynomials is not finitely generated.
Recently the authors \cite{FO} have provided an infinite set of generators for $F_2^{S_2}$, later
the result was generalized in \cite{DFO}.

One may consider the group $\text{\rm Aut}(F_n^{S_n})$ of automorphisms
preserving the algebra $F_n^{S_n}$. The group  $\text{\rm Aut}(F_n)$
is a semidirect product of the general linear group $ \text{\rm GL}_n(K)$ 
and the group $\text{\rm IAut}(F_n)$ of automorphisms which are equivalent to the identity map modulo the commutator ideal
$F_n'$. Hence it is natural to work in $\text{\rm IAut}(F_n)$ approaching the group $\text{\rm Aut}(F_n^{S_n})$.
However the complete decription of the group $\text{\rm IAut}(F_n)$ is unknown, while its normal subgroup $\text{\rm Inn}(F_n)$
of inner automorphisms is well known.

In this study, we determine the group of inner automorphisms of $F_n^{S_n}$.
Additionally we describe the group $\text{\rm Inn}(L_{n,c}^{S_n})$ of inner automorphisms,
where $L_{n,c}$ is the free metabelian Lie algebra  of nilpotency class $c$.
Later, we describe the groups $\text{\rm Aut}(L_2^{S_2})$, $\text{\rm Aut}(F_2^{S_2})$ and $\text{\rm Aut}(L_{2,c}^{S_2})$,
where $L_2$ is the free Lie algebra of rank $2$.

\section{Preliminaries}

Let $L_n$ be the free Lie algebra of rank $n\geq2$ generated by $X_n=\{x_1,\ldots,x_n\}$ over a field $K$ of characteristic zero.
We denote by $F_n=L_n/L_n''$ the free metabelian Lie algebra,
and $L_{n,c}=L_n/(L_n''+\gamma^{c+1}(L_n))$ the free metabelian nilpotent Lie algebra of nilpotency class $c$, where
$\gamma^1(L_n)=L_n$, $\gamma^2(L_n)=[L_n,L_n]=L_n'$ is the commutator ideal of $L_n$, $\gamma^k(L_n)=[\gamma^{k-1}(L_n),L_n]$, $k\geq2$, and $L_n''=[L_n',L_n']$.
We assume that the algebras $F_n$ and $L_{n,c}$ of rank $n$ are generated by the same set $X_n$.

The commutator ideal $F_n'$ of the free metabelian Lie algebra $F_n$ is of a natural
$K[X_n]$-module structure as a consequence of the metabelian identity
\[
[[z_1,z_2],[z_3,z_4]]=0, \ \ z_1,z_2,z_3,z_4\in F_n,
\]
with action:
\[
f(X_n)g(x_1,\ldots,x_n)=f(X_n)g(\text{\rm ad}x_1,\ldots,\text{\rm ad}x_n), \ \ f(X_n)\in F_n', \ \ g(X_n)\in K[X_n],
\]
where $K[X_n]$ is the (commutative, associative, unitary) polynomial algebra,
and the adjoint action is defined as $z_1\text{ad}z_2=[z_1,z_2]$, for $z_1,z_2\in F_n$.
One may define a similar action on the free metabelian nilpotent Lie algebra $L_{n,c}$. 
It is well known by Bahturin \cite{Ba} that the monomials
$[x_{k_1},x_{k_2}]x_{k_3}\cdots x_{k_l}$, $k_1>k_2\leq k_3\leq k_l$,
form a basis for $F_n'$, which is a basis for $L_{n,c}'$ when $l\leq c-2$.

An element $s(X_n)$ in $L_n$, $F_n$, or $L_{n,c}$ is called symmetric
if
\[
s(x_1,\ldots,x_n)=s(x_{\pi(1)},\ldots,x_{\pi(n)})=\pi s(x_1,\ldots,x_n)
\]
for each permutation $\pi$
in the symmetric group $S_n$. The sets $L_n^{S_n}$, $F_n^{S_n}$, and $L_{n,c}^{S_n}$ of symmetric
polynomials coincide with the algebras of invariants of the group $S_n$. See the work \cite{FO} for
generators of the algebra $F_2^{S_2}$, and its generalization \cite{DFO} for the full description of the algebra
$F_n^{S_n}$. The algebra $L_{n,c}^{S_n}$ is a direct consequence of the known results on the algebra
$F_n^{S_n}$.

Let $A_n$ stand for $L_n$, $F_n$, or $L_{n,c}$. 
It is well known that the automorphism group $\text{\rm Aut}(A_n)$ is a semidirect product
of the general linear group $ \text{\rm GL}_n(K)$ 
and the group $\text{\rm IAut}(A_n)$ of automorphisms which are equivalent to the identity map modulo the commutator ideal
$A_n'$. Hence it is natural to work in $\text{\rm IAut}(A_n)$
when determining the whole group $\text{\rm Aut}(A_n)$.
Now consider the group $\text{\rm Aut}(A_n^{S_n})$ of automorphisms
consisting of automorphisms of $A_n$ preserving each symmetric polynomial in the algebra $A_n^{S_n}$.

In the next section, as an approach to the group $\text{\rm IAut}(A_n^{S_n})$,
we describe the inner automorphism group $\text{\rm Inn}(A_n^{S_n})$ for $A_n=F_n$, and $A_n=L_{n,c}$.
The case $\text{\rm Inn}(L_n^{S_n})$ is not under the consideration, since the free Lie algebra $L_n$ does not have nontrivial inner automorphisms.
Later, we obtain the groups $\text{\rm Aut}(L_2^{S_2})$,
$\text{\rm Aut}(F_2^{S_2})$, and $\text{\rm Aut}(L_{2,c}^{S_2})$ as  a consequence of the results obtained.

\section{Main Results}

\subsection{Inner automorphisms of $F_n^{S_n}$}

Let $u\in F_n'$ be an element from the commutator ideal of the free metabelian Lie algebra $F_n$.
Then the adjoint operator $\text{\rm ad}u:v\to [v,u]$ is a nilpotent derivation of $F_n$, and
$\psi_u=\exp(\text{\rm ad}u)=1+\text{\rm ad}u$
is an automorphism of the Lie algebra $F_n$. The inner automorphism group $\text{\rm Inn}(F_n)$
of $F_n$ consisting of such automorphisms is abelian: $\psi_{u_1}\psi_{u_2}=\psi_{u_1+u_2}$, $\psi_u^{-1}=\psi_{-u}$.

In the following theorem we determine the group $\text{\rm Inn}(F_n^{S_n})$ of inner automorphisms preserving the algebra $F_n^{S_n}$.

\begin{theorem}\label{metabelian}
The automorphism $\psi_u\in\text{\rm Inn}(F_n^{S_n})$ if and only if $u\in (F_n')^{S_n}$.
\end{theorem}

\begin{proof}
If $u\in (F_n')^{S_n}$, then clearly $\psi_u(v)=v+[v,u]\in F_n^{S_n}$ for every $v\in F_n^{S_n}$.
Conversely let $v\in F_n^{S_n}$ be a symmetric polynomial, and $u\in F_n'$ be an arbitrary element. We may assume that the linear (symmetric) summand $v_l$ of $v$ is nonzero,
since  $\psi_u$ acts identically on the commutator ideal $ F_n'$ of the free metabelian Lie algebra $F_n$.
Then $\psi_u(v)\in F_n^{S_n}$ implies that $[v,u]=[v_l,u]\in F_n^{S_n}$. 
For each $\pi\in S_n$, we have that
\[
[v_l,u]=\pi[v_l,u]=[\pi  v_l,\pi u]=[v_l,\pi u],
\]
and $[v_l,u-\pi u]=0$, which gives that $u-\pi u=0$ or $u=\pi u$.
\end{proof}

\subsection{Inner automorphisms of $L_{n,c}^{S_n}$}

Let $u$ be an element in the free metabelian nilpotent Lie algebra $L_{n,c}$.
Then the adjoint operator $\text{\rm ad}u(v)=[v,u]$ is a nilpotent derivation of $L_{n,c}$, since $\text{\rm ad}^cu=0$ and
\[
\varepsilon_u=\exp(\text{\rm ad}u)=1+\text{\rm ad}u+\frac{1}{2}\text{\rm ad}^2u+\cdots+\frac{1}{(c-1)!}\text{\rm ad}^{c-1}u
\]
is an inner automorphism of $L_{n,c}$. Note that if $u\in\gamma^c(L_{n,c})$, then $\varepsilon_u$ acts
identically on $L_{n,c}$, thus we may define the
inner automorphism group of $L_{n,c}$ as follows.
\[
\text{\rm Inn}(L_{n,c})=\{\varepsilon_u\mid u\in L_{n,c}-\gamma^c(L_{n,c})\}\cup\{1\}.
\]
In this subsection, we investigate the group $\text{\rm Inn}(L_{n,c}^{S_n})$
of inner automorphisms of the algebra $L_{n,c}^{S_n}$.

\begin{lemma}\label{lineartimeslinear}
Let $u=\sum_{i=1}^n\alpha_ix_i$ for some $\alpha_i\in K$, and $v=\sum_{i=1}^nx_i\in L_{n,c}^{S_n}$ such that $[u,v]\in L_{n,c}^{S_n}$.
Then $u=\alpha v$ for some $\alpha\in K$.
\end{lemma}

\begin{proof}
Let $\pi=(1k)\in S_n$ be a fixed transposition for $k=2,\ldots,n$. Then
\[
\pi u=\alpha_1x_{\pi(1)}+\cdots+\alpha_nx_{\pi(n)}=\alpha_1x_k+\alpha_kx_1+\sum_{i\neq1,k}\alpha_ix_i,
\]
and $u-\pi u=\alpha_1(x_1-x_k)+\alpha_k(x_k-x_1)=\alpha_{1k}(x_1-x_k)$,
where $\alpha_{1k}=\alpha_1-\alpha_k$. Now $[u,v]\in L_{n,c}^{S_n}$ gives that
$[u,v]=\pi[u,v]=[\pi u,\pi v]=[\pi u,v]$,
and hence
\begin{align}
0&=[u-\pi u,v]=[\alpha_{1k}(x_1-x_k),x_1+\cdots+x_n]\nonumber\\
&=\alpha_{1k}\left(2[x_{1},x_{k}]+\sum_{i\neq 1,k}[ x_1,x_i]+\sum_{i\neq 1,k}[x_i,x_k]\right)\nonumber
\end{align}
where the elements in the paranthesis are basis elements of $L_{n,c}$,
which implies that $\alpha_{1k}=0$, $k\geq2$. This completes the proof by the choice $\alpha=\alpha_1=\cdots=\alpha_n$.
\end{proof}

\begin{theorem}\label{metabeliannilpotent}
$\text{\rm Inn}(L_{n,c}^{S_n})=\{\varepsilon_u\mid u\in  L_{n,c}^{S_n}-\gamma^c(L_{n,c})^{S_n}\}\cup\{1\}$.
\end{theorem}

\begin{proof}
If $u,v\in L_{n,c}^{S_n}$ then it is straighforward to see  that
\[
\varepsilon_u(v)=v+[v,u]+\cdots+(1/(c-1)!)[[\cdots[v,u],\ldots],u]\in L_{n,c}^{S_n}.
\]
Conversely, let $u=u_l+u_0\in L_{n,c}$ be an arbitrary element 
and $v=v_l+v_0\in L_{n,c}^{S_n}$ be a symmetric polynomial such that $\varepsilon_u(v)\in L_{n,c}^{S_n}$,
where $u_l$ and $v_l$ are the linear components of $u$ and $v$, respectively.
In the expression of $\varepsilon_u(v)$, the component of degree $2$ is $[v_l,u_l]$
which is symmetric by the natural grading on the Lie algebra $L_{n,c}^{S_n}$. 
Hence $u_l=\alpha v_l$ for some $\alpha\in K$ by Lemma \ref{lineartimeslinear}, and $[v_l,u_l]=0$.
Note that $[v_0,u_0]=0$ by metabelian identity. The computations
\begin{align}
\varepsilon_u(v)=&v+[v_l+v_0,u_l+u_0]\sum_{k=0}^{c-2}\frac{u_l^k}{(k+1)!}\nonumber\\
=&v+[v_0,v_l]\sum_{k=0}^{c-3}\frac{\alpha^{k+1}v_l^k}{(k+1)!}+[v_l,u_0]\sum_{k=0}^{c-3}\frac{\alpha^kv_l^k}{(k+1)!}\nonumber
\end{align}
give that $[v_l,u_0]\sum_{k=0}^{c-3}\frac{\alpha^kv_l^k}{(k+1)!}\in L_{n,c}^{S_n}$, and that $[v_l,u_0]$ is symmetric.
Finally, following the lines of proof of Theorem \ref{metabelian}, by similar computations we obtain that $u_0\in L_{n,c}^{S_n}-\gamma^c(L_{n,c})^{S_n}$.
\end{proof}

\subsection{Automorphisms of $L_2^{S_2}$, $F_2^{S_2}$, and $L_{2,c}^{S_2}$}

In the sequel, we fix the notation $x_1=x$, $x_2=y$, for the sake of simplicity.
It is well known by \cite{C} that each automorphism of $L_2$ is linear. The next theorem determines the automorphism group $\text{\rm Aut}(L_2^{S_2})$.

\begin{theorem}\label{linear}
Let $\xi\in \text{\rm Aut}(L_2^{S_2})$. Then $\xi$ and its inverse $\xi^{-1}$ are of the form
\[
\xi(x)=ax+by, \ \ \xi(y)=bx+ay,
\]
\[
\xi^{-1}(x)=c^{-1}ax-c^{-1}by,\ \ \xi^{-1}(y)=-c^{-1}bx+c^{-1}ay,
\]
such that $c=a^2-b^2\neq 0$, $a,b\in K$.
\end{theorem}

\begin{proof}
Let $\xi$ be of the form $\xi:x\to ax+by, \ y\to cx+dy$
such that $ad\neq bc$, where $a,b,c,d\in K$.
Since $x+y\in L_2^{S_2}$, then
$\xi(x+y)=(a+c)x+(b+d)y\in L_2^{S_2}$,
which is contained in $K\{x+y\}$. Hence $a+c=b+d$.  On the other hand
$[[x,y],x]-[[x,y],y]\in L_2^{S_2}$, and
\[
\xi([[x,y],x]-[[x,y],y])=\beta\left((a-c)[[x,y],x]+(b-d)[[x,y],y]\right)
\]
where $\beta=ad-bc\neq 0$. The fact that $\xi\in \text{\rm Aut}(L_2^{S_2})$ gives $a-c=-b+d$. Consequently, $a=d$ and $b=c$.
Conversely, it is straightforward to show that the automorphism stated in the theorem preserves symmetric polynomials.
\end{proof}

It is well known, see the book by Drensky \cite{Dr}, that
each automorphism of $F_2$ is a product of a linear automorphism and an inner automorphism of $F_2$.
We obtain the following result as a consequence of Theorem \ref{linear} and Theorem \ref{metabelian}.

\begin{corollary}
Let $\varphi\in \text{\rm Aut}(F_2^{S_2})$. Then $\varphi$ is a product of a linear automorphism
 $\xi=\xi(a,b)$ as in Theorem \ref{linear} and an inner automorphism $\psi_u$ where $u\in (F_2')^{S_2}$.
\end{corollary}

\begin{theorem}
Let $\varphi\in \text{\rm Aut}(L_{2,c}^{S_2})$. Then $\varphi$ is a product of a linear automorphism
$\xi(a,b)$ and an automorphism $\phi\in\text{\rm IAut}(L_{2,c}^{S_2})$of the form
\begin{align}
\phi:&x\to x+[x,y]f(x,y)\nonumber\\
&y\to y-[x,y]f(y,x).\nonumber
\end{align}
\end{theorem}

\begin{proof}
It is sufficient to show that an automorphism $\phi\in \text{\rm IAut}(L_{2,c})$ preserving symmetric polynomials
satisfies the condition of the theorem. In general $\phi$ is of the form
$\phi(x)=x+[x,y]f(x,y)$, $\phi(y)=y+[x,y]g(x,y)$.
Since $x+y\in L_{2,c}^{S_2}$, then $\phi(x+y)$ is symmetric, and hence
\[
x+y+[x,y](f(x,y)+g(x,y))=y+x-[x,y](f(y,x)+g(y,x)).
\]
This gives that
\begin{align}\label{equation}
f(x,y)+g(x,y)+f(y,x)+g(y,x)=0
\end{align}
in the commutator ideal $L_{2,c}'$ of $L_{2,c}$,
which is a $K[x,y]$-module freely generated by $[x,y]$. Now by the symmetric polynomial $[x,y](x-y)$,
we have that
\[
[x+[x,y]f(x,y),y+[x,y]g(x,y)](x-y)=[x,y](x-y)(1+f(x,y)y-xg(x,y))
\]
is symmetric. Consequently $[x,y](x-y)(yf(x,y)-xg(x,y))$ is symmetric. Hence,
\[
[x,y](x-y)(yf(x,y)-xg(x,y))=[x,y](x-y)(xf(y,x)-yg(y,x)),
\]
and thus using Equation \ref{equation} we have that
\begin{align}
0&=yf(x,y)-xg(x,y)-xf(y,x)+yg(y,x)\nonumber\\
&=-xg(x,y)-xf(y,x)+y(f(x,y)+g(y,x))\nonumber\\
&=-xg(x,y)-xf(y,x)-y(g(x,y)+f(y,x))\nonumber\\
&=-(x+y)(f(y,x)+g(x,y)).\nonumber
\end{align}
which implies that $g(x,y)=-f(y,x)$.
\end{proof}

\end{document}